\documentclass[11pt,reqno]{amsart}

\usepackage{amsmath,amsthm,amssymb,comment,fullpage}
\usepackage{braket}
\usepackage{mathtools}

\usepackage{caption}
\usepackage{times}
\usepackage[T1]{fontenc}
\usepackage{mathrsfs}
\usepackage{svg}
\usepackage{latexsym}
\usepackage{epsfig}
\usepackage{amsmath,amsfonts,amsthm,amssymb,amscd}
\input amssym.def
\input amssym.tex
\usepackage{color}
\usepackage{hyperref}
\usepackage{url}
\newcommand{\bburl}[1]{\textcolor{blue}{\url{#1}}}

\usepackage{tikz}
\usepackage{tkz-tab}
\usetikzlibrary{shapes.geometric,positioning}

\usepackage{scalerel}
\def\shrinkage{-2.4mu}
\def\vecsign#1{\rule[1.388\LMex]{\dimexpr#1-2.5pt}{.36\LMpt}%
  \kern-6.0\LMpt\mathchar"017E}
\def\dvecsign#1{\rule{0pt}{7\LMpt}\smash{\stackon[-1.989\LMpt]{%
  \SavedStyle\mkern-\shrinkage\vecsign{#1}}%
  {\rotatebox{180}{$\SavedStyle\mkern-\shrinkage\vecsign{#1}$}}}}
\def\dvec#1{\ThisStyle{\setbox0=\hbox{$\SavedStyle#1$}%
  \def\useanchorwidth{T}\stackon[-4.2\LMpt]{\SavedStyle#1}{\,\dvecsign{\wd0}}}}
\usepackage{stackengine,amsmath}
\stackMath
\usepackage{graphicx}

\newcommand{\burl}[1]{\textcolor{blue}{\url{#1}}}

\numberwithin{equation}{section}

\newtheorem{thm}{Theorem}[section]
\newtheorem{conj}[thm]{Conjecture}
\newtheorem{cor}[thm]{Corollary}

\newtheorem{prop}[thm]{Proposition}

\theoremstyle{plain}

\newtheorem{definition}[thm]{Definition}

\newtheorem{lemma}[thm]{Lemma}

\newtheorem{theorem}[thm]{Theorem}

\newcommand\be{\begin{equation}}
\newcommand\ee{\end{equation}}
\newcommand\bee{\begin{equation*}}
\newcommand\eee{\end{equation*}}
\newcommand\bea{\begin{eqnarray}}
\newcommand\eea{\end{eqnarray}}
\newcommand\beae{\begin{eqnarray*}}
\newcommand\eeae{\end{eqnarray*}}
\newcommand\bi{\begin{itemize}}
\newcommand\ei{\end{itemize}}
\newcommand\ben{\begin{enumerate}}
\newcommand\een{\end{enumerate}}
\newcommand\bc{\begin{center}}
\newcommand\ec{\end{center}}
\newcommand\ba{\begin{array}}
\newcommand\ea{\end{array}}

\newcommand\frakfamily{\usefont{U}{yfrak}{m}{n}}
\DeclareTextFontCommand{\textfrak}{\frakfamily}

\newcommand{\hr}[1]{\href{#1}{\url{#1}}}

\title{Optimal Point Sets Determining Few Distinct Angles}

\author{Henry L. Fleischmann}
\email{\textcolor{blue}{\href{mailto:henryfl@umich.edu}{henryfl@umich.edu}}}
\address{Department of Mathematics, University of Michigan, Ann Arbor, MI 48109}

\author{Steven J. Miller}
\email{\textcolor{blue}{\href{mailto:sjm1@williams.edu}{sjm1@williams.edu}},  \textcolor{blue}{\href{Steven.Miller.MC.96@aya.yale.edu}{Steven.Miller.MC.96@aya.yale.edu}}}
\address{Department of Mathematics and Statistics, Williams College, Williamstown, MA 01267}

\author{Eyvindur A. Palsson}
\email{\textcolor{blue}{\href{mailto:palsson@vt.edu}{palsson@vt.edu}}}
\address{Department of Mathematics, Virginia Tech, Blacksburg, VA 24061}

\author{Ethan Pesikoff}
\email{\textcolor{blue}{\href{mailto:ethan.pesikoff@yale.edu}{ethan.pesikoff@yale.edu}}}
\address{Department of Mathematics, Yale University, New Haven, CT 06511}

\author{Charles Wolf}
\email{\textcolor{blue}{\href{mailto:charles.wolf@rochester.edu}{charles.wolf@rochester.edu}}}
\address{Department of Mathematics, Rochester, NY, 14627}

\thanks{This work was supported by NSF grant 1947438 and Williams College. E. A. Palsson was supported in part by Simons Foundation grant $\#$360560.}



\date{\today}

\begin{document}

\maketitle

\begin{abstract}
    Let $P(k)$ denote the largest size of a non-collinear point set in the plane admitting at most $k$ angles. We prove $P(1)=3$, $P(2)=5$ and $P(3)=5$, and we characterize the optimal sets.  We also leverage results from \cite{GenPaper} in order to provide the general bounds of $k+2 \leq P(k) \leq 6k$, although the upper bound may be improved pending progress toward the Weak Dirac Conjecture.  Notably, it is surprising that $P(k)=\Theta(k)$ since, in the distance setting, the best known upper bound on the analogous quantity is quadratic, and no lower bound is well-understood.
\end{abstract}

\tableofcontents

\section{Introduction}
\subsection{Background}
In 1946, Erdős introduced the problem of finding asymptotic bounds on the minimum number of distinct distances among sets of $n$ points in the plane \cite{ErOg}. The Erd\H{o}s distance problem, as it has become known, proved infamously difficult and was only finally (essentially) resolved by Guth and Katz in 2015 \cite{GuthKatz}.

The Erd\H{o}s distance problem has also spawned a wide variety of related questions, including the problem of finding maximal point sets with at most $k$ distinct distances. Erd\H{o}s and Fishburn determine maximal planar sets with at most $k$ distinct distances \cite{ErFi}. Recent results by Szöll\H{o}si and Östergård classify the maximal 3-distance sets in $\mathbb{R}^4$, 4-distance sets in $\mathbb{R}^3$, and 6-distance sets in $\mathbb{R}^2$ \cite{Xi,SzOs}. In \cite{ELMP, BrDePaSe, BrDePaSt} point sets with a low number of distinct triangles in Euclidean space are investigated. In \cite{GenPaper}, a number of angle analogues of distinct distance problems are considered.  Recently, new connections to frame theory and engineering have renewed interested in few-distance sets \cite{SzOs}.

Characterizing the largest possible point sets satisfying a given property in this way is a classic problem in discrete geometry. As another example, Erd\H{o}s introduced the problem of finding maximal point sets of all isosceles triangles in 1947 \cite{ErKeIsosSets}. Ionin completely answers this question in Euclidean space of dimension at most $7$ \cite{Io}.

We study one variation of a related problem of Erdős and Purdy \cite{ErPur}. They asked about $A(n)$, the minimum number of distinct angles formed by $n$ not-all-collinear points in the plane.  Recently, \cite{GenPaper} made partial progress on this problem, and the best known bounds are $n/6\leq A(n)\leq n-2$.  We consider the related problem of maximal planar point sets admitting at most $k$ distinct angles in $(0, \pi)$.  We ignore angles of $0$ and $\pi$ so as to align the convention in related research (see \cite{PaSha}, for example), although we provide results including the 0 angle as corollaries. We completely answer this question for $k=2$, and $k=3$ and note that the work from \cite{GenPaper} immediately implies asymptotically tight linear bounds for $k > 3$. In answering this question for $k=2$ and $k=3$, we systematically consider all possible triangles in such configurations and then reduce to adding points in a finite number of positions by geometric casework.  We thus both find $P(2),P(3)$ and classify all optimal configurations.
 
\subsection{Definitions and Results}
By convention, we only count angles of magnitude strictly between $0$ and $\pi$. Our computations still answer the related optimal point configuration questions including $0$ angles (see Corollaries \ref{cor: p2-with-0}, \ref{cor: p3-with-0}).
We begin by introducing convenient notation:
\begin{definition}
Let $\mathcal{P} \subset \mathbb{R}^2$. Then
\[
A(\mathcal{P}) 
\coloneqq \#\{|\angle abc|\in(0,\pi) \,:\, a,b,c \text{ distinct, } a,b,c \in \mathcal{P}\},
\]
\end{definition}

Now we define the quantity we are interested in studying.
\begin{definition}
\[
P(k) \coloneqq \max \{\# \mathcal{P} \,:\, \mathcal{P} \subseteq \mathbb{R}^2, \text{ not all points in }\mathcal{P}\text{ are collinear, } A(\mathcal{P}) \leq k\}. 
\]
\end{definition}

We first provide general linear lower and upper bounds for $P(k)$. In particular, we have the following theorem.

\begin{theorem}\label{thm: linear bounds on P(k)}
For all $k \geq 1$,
\begin{alignat*}{2}
2k+3 \, &\leq \, P(2k)  &&\leq \, 12k \\
2k+3 \, &\leq \, P(2k+1) &&\leq \, 12k + 6.
\end{alignat*}
\end{theorem}

In the distance setting, the best known upper bound on the analogous parameter is the quadratic $(2+k)(1+k)$, and no lower bound is well-understood \cite{SzOs}.  It is therefore interesting and surprising that we find $P(k)=\Theta(k)$ in the angle setting.  We prove Theorem \ref{thm: linear bounds on P(k)} in Section \ref{sec: gen bds}.

Furthermore, we explicitly compute $P(1), P(2)$, and $P(3)$ and exhaustively identify all maximal point configurations for each.

\begin{prop} \label{thm: P(1) = 3}
    We have $P(1) = 3$, and the equilateral triangle is the unique maximal configuration.
\end{prop}

In order to have only a single angle, every triangle of three points in the configuration must be equilateral. As this is impossible for point configurations that are not the vertices of an equilateral triangle, $P(1) = 3$. $P(2)$ and $P(3)$ are considerably less trivial quantities. We calculate $P(2), P(3)$ via exhaustive casework, simultaneously characterizing all of the unique optimal point configurations up to rigid motion transformations and dilation about the center of the configuration.  We proceed by first considering sets of three points and then search for what additional points may be added without determining too many angles. We prove Theorem \ref{thm: P(2) = 5} in Section \ref{sec: P(2)} and Theorem \ref{thm: P(3) = 5} in Section \ref{sec: P(3)}.

\begin{theorem} \label{thm: P(2) = 5}
    We have $P(2) = 5$. Moreover, the unique optimal point configuration is four vertices in a square with a fifth point at its center (see A in Figure \ref{fig: 3.thm}).
\end{theorem}

\begin{theorem} \label{thm: P(3) = 5}
    We have $P(3) = 5$. There are 5 unique optimal configurations, shown in Figure \ref{fig: 3.thm}.
\end{theorem}
\begin{figure}[h!] \centering
\includegraphics{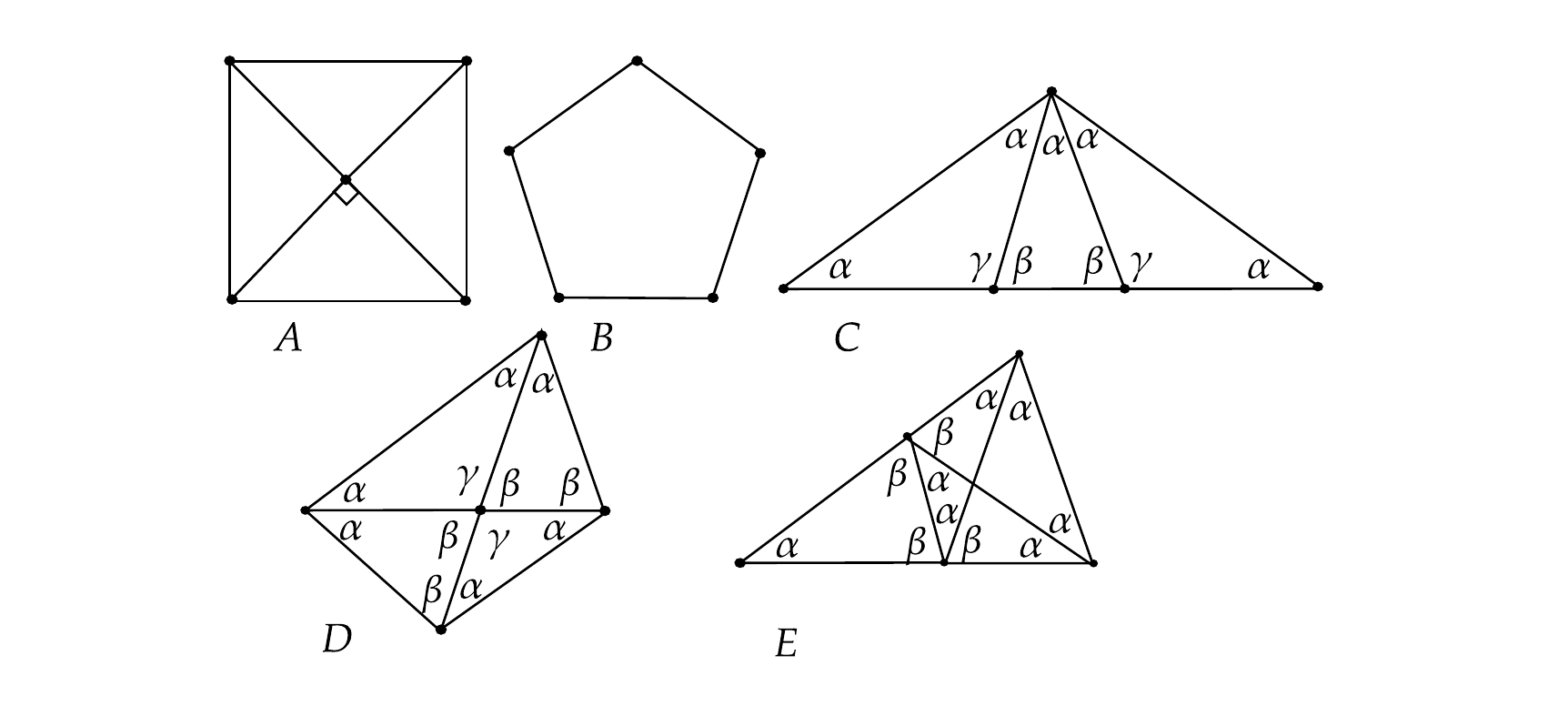}
\caption{Optimal Two and Three Angle Configurations. $\alpha = \frac{\pi}{5}, \beta = \frac{2\pi}{5}, \gamma = \frac{3\pi}{5}$.} \label{fig: 3.thm}\end{figure}


\section{General Bounds} \label{sec: gen bds}
Although one may in principle calculate $P(k)$ for any $k$ by extensive casework (as we later calculate $P(2),P(3)$), it quickly becomes overwhelming. As such, we instead provide general bounds on $P(k)$.  In \cite{GenPaper} the authors study the quantity $A(n)$, the minimum number of angles admitted by a non-collinear point set of $n$ points in the plane.  They show in Lemma 2.2 and Theorem 2.5 that $n/6\leq A(n)\leq n-2$, noting that the lower bound may be improved up to as much as $n/4-1$, pending progress on the Weak Dirac Conjecture.  Since $A(n)\leq n-2$, then $n\geq A(n)+2$, and so we deduce that $P(k)\geq k+2$.  Similarly, we have $P(k)\leq 6k$.  Combining these bounds gives the desired result
\begin{prop}
    $k+2\leq P(k)\leq 6k$.
\end{prop}

Notably, it is surprising that $P(k)=\Theta(k)$ since, in the distance setting, the best known upper bound on the analogous quantity is quadratic, and no lower bound is well-understood.

\section{Computing $P(2) = 5$} \label{sec: P(2)}

\begin{proof}
In any point configuration with at least three points, there are triangles. For any point configuration with at most two angles, all triangles must be isosceles.  We divide into two cases, based on whether or not there is an equilateral triangle.

\subsection{There is an equilateral triangle}
We consider adding a fourth point in cases (Figure \ref{2.equil}).
\begin{figure}[h!] \centering
\includegraphics[scale=0.8]{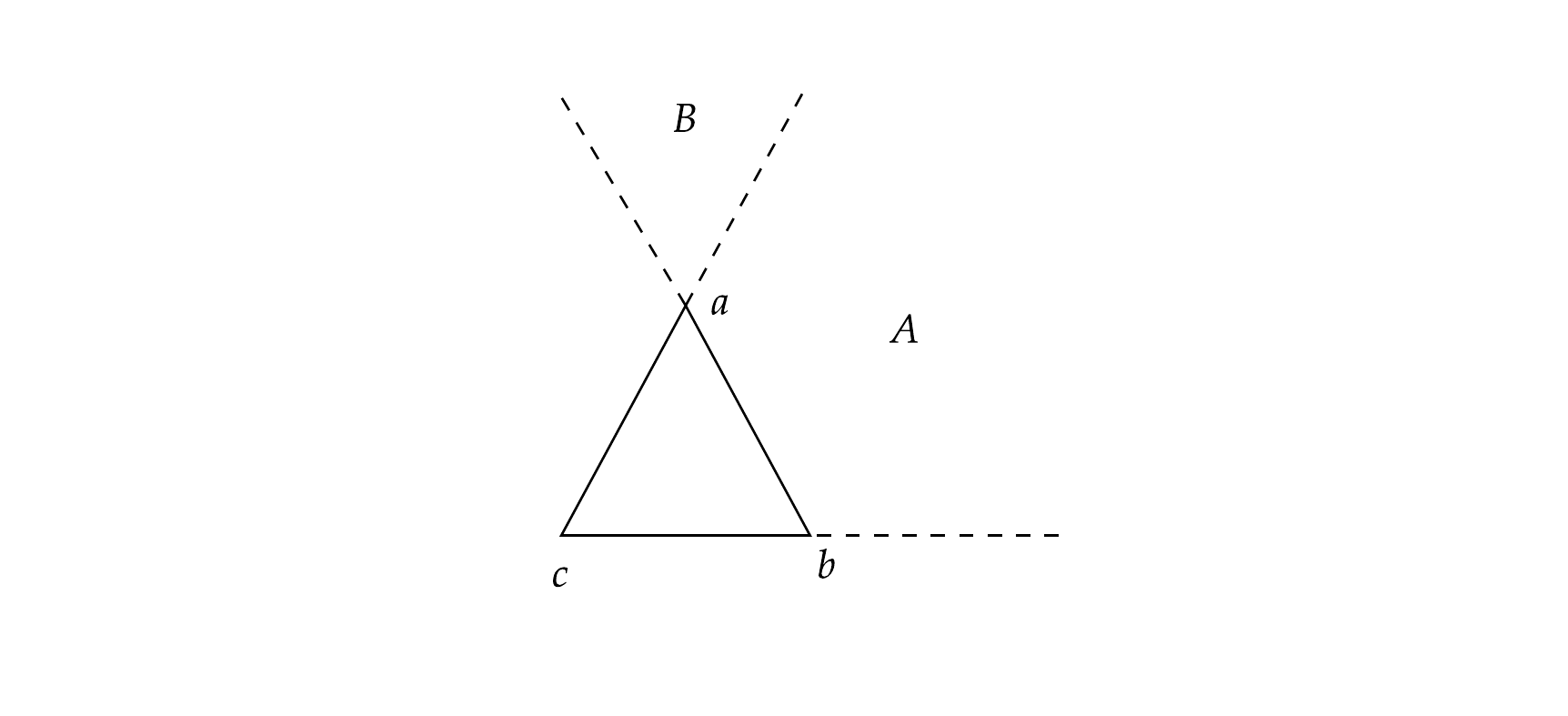}
\caption{Equilateral Triangle Regions} \label{2.equil} \end{figure}

\begin{description}
\item[\textit{Case 1}] $p \in A$.\\
Then $\angle acp < \pi/3$ and $\angle cap > \pi/3$, leading to more than two angles.
\item[\textit{Case 2}] $p \in \overline{ab}$.\\
Then $\angle bcp < \pi/3$ and one of $\angle cpb$ and $\angle apc \geq \pi/2$, leading to more than two angles. 
\item[\textit{Case 3}] $p \in \dvec{ac}$ to the upper-right of $a$.\\
Then $\angle cbp > \pi/3$ and $\angle cpb < \pi/3$, again leading to more than two angles. 
\item[\textit{Case 4}] $p \in B$.\\
In this case, $\angle cbp > \pi/3$ and $\angle cpb < \pi/3$, leading to more than two angles.
\item[\textit{Case 5}] $p \in \triangle{abc}$.\\
In this case, one of $\angle apb, \angle bpc, \angle cpa \geq 2\pi/3$ and $\angle acp < \pi/3$, leading to more than two angles.
\end{description}

Up to symmetry, these cases are exhaustive. Thus if there is an equilateral triangle in the configuration, there can only be at most three points.

\subsection{There is no equilateral triangle}

Now, let $a$, $b$, and $c$ be the vertices of an isosceles triangle with vertex angle $\alpha$, base angle $\beta$, and $a$ the apex vertex.
We reduce the number of possibilities for additional points by partitioning the plane into regions $A_i$ (Figure \ref{2.regions}).
\begin{figure}[h!] \centering
\includegraphics{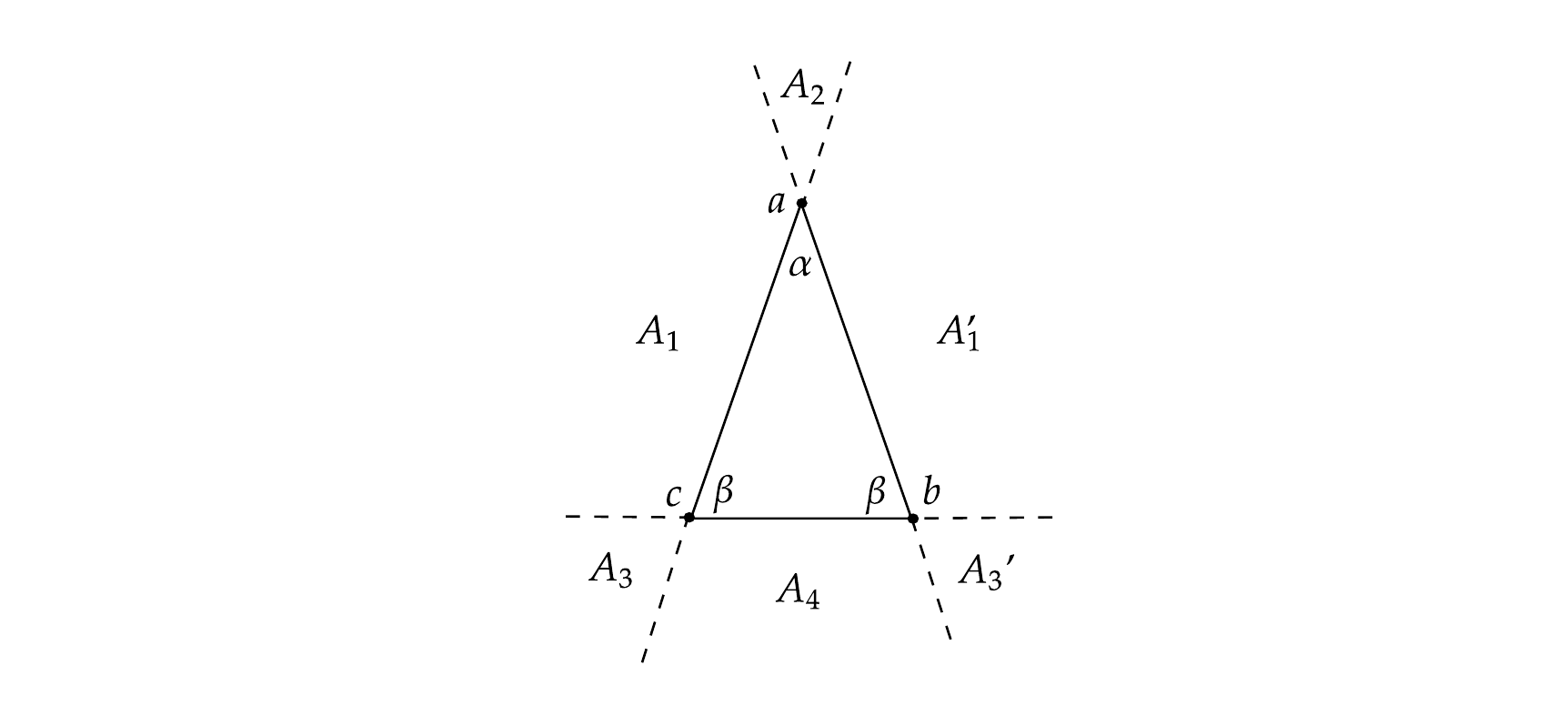}
\caption{Isosceles Triangle Regions.} \label{2.regions} \end{figure}
Note that we may without loss of generality assume that no fourth point is added within $\triangle abc$ as we could then choose one of the resultant interior triangles as our initial triangle. Also note that $A_1$ and $A_1'$ and $A_3$ and $A_3'$ are equivalent up to symmetry.
\begin{description}
\item[\textit{Case 1}] $p \in A_1$.\\
In this case, $\angle pab > \alpha$ and $\angle pcb > \beta$. So, regardless of whether $\alpha$ or $\beta$ is greater, adding $p$ introduces an additional angle. So, no additional points can be in $A_1$ or $A_1'$. 

\item[\textit{Case 2}] $p \in A_2$.\\
In this case, $\angle pcb$ and $\angle pbc$ are greater than $\beta$, so both must be $\alpha$ to not add additional angles. But then $\angle cpb = \pi - 2\alpha \neq \beta$, in order to not add angles, implying $3 \alpha = \pi$. But, this implies $\triangle pcb$ is an equilateral triangle. Thus no points may be added in this case.

\item[\textit{Case 3}] $p \in A_3$ (or $A_3'$ by symmetry). \\
In this case, $\angle bap > \alpha$ and $\angle abp > \beta$, so there is an additional angle added regardless and no additional points are possible.

\item[\textit{Case 4}] $p \in A_4$. \\
In this case, $\angle cap, \angle bap < \alpha$, so both must equal $\beta$. Therefore, $2 \beta = \alpha$, which implies $\beta = \pi/4$ and $\alpha = \pi /2$. Moreover, since $\angle acp$ and $\angle abp$ are greater than $\beta$, they must both equal $\alpha = \pi/2$. So, the only possibility for an addable point in this case is for $p$ to be the fourth vertex of the square $acpb$.

\item[\textit{Case 5}] $p \in \dvec{bc}$. \\
If $p$ is on $\dvec{bc}$ between $b$ and $c$, then $\angle cap, \angle bap < \alpha$. In order for these not to introduce additional angles, they must both be equal to $\beta$. This implies $\beta = \pi/4$ and $\alpha = \pi/2$ and $p$ is the center of the side $bc$. If $p \in \dvec{bc}$ to the left of $c$ (or by symmetry, right of $b$), $\angle bap > \alpha$ and thus $\angle bap = \beta$. Since $2\beta + \alpha = \pi$, $\beta < \pi/2$. But then $\angle acp > \pi/2 > \beta > \alpha$. Thus there is exactly one point possible on line $\dvec{bc}$, the centerpoint of the edge between $b$ and $c$.

\item[\textit{Case 6}] $p \in \dvec{ac}$ (or $p \in \dvec{ab}$).\\
If $p$ is between $a$ and $c$, then $\angle cbp < \beta$ and thus $\angle cbp = \alpha$. But, as before, $\beta < \pi/2$. Moreover, one of $\angle bpc$ or $\angle bpa$ is at least $\pi/2 > \beta > \alpha$. Thus there are too many angles in this case. If $p$ is to the bottom left of $c$, $\angle apb < \beta$ and thus $\angle apb = \alpha$. But, again, either $\angle bca$ or $\angle bcp > \pi/2 > \beta$, creating too many angles in this case. If $p$ is on $\dvec{ac}$ to the upper right of $a$, $\angle pbc > \beta$ and thus equals $\alpha$. Then $\angle pba < \alpha$ and must equal $\beta$ and thus $2\beta = \alpha$. This implies $\beta = \pi/4$ and $\alpha = \pi/2$ and $\triangle cbp$ is an isosceles right triangle with $b$ the apex vertex, $p$ on $\dvec{ac}$ to the upper right of $a$, and $a$ at the center of side $\overline{pc}$.
\end{description}

As such, in order to add additional points to an isosceles triangle point configuration without adding additional angles, we must have $\alpha = \pi/2$ and $\beta = \pi/4$. The four additional possible points are marked in Figure \ref{2.points}.
\begin{figure}[h!] \centering
\includegraphics{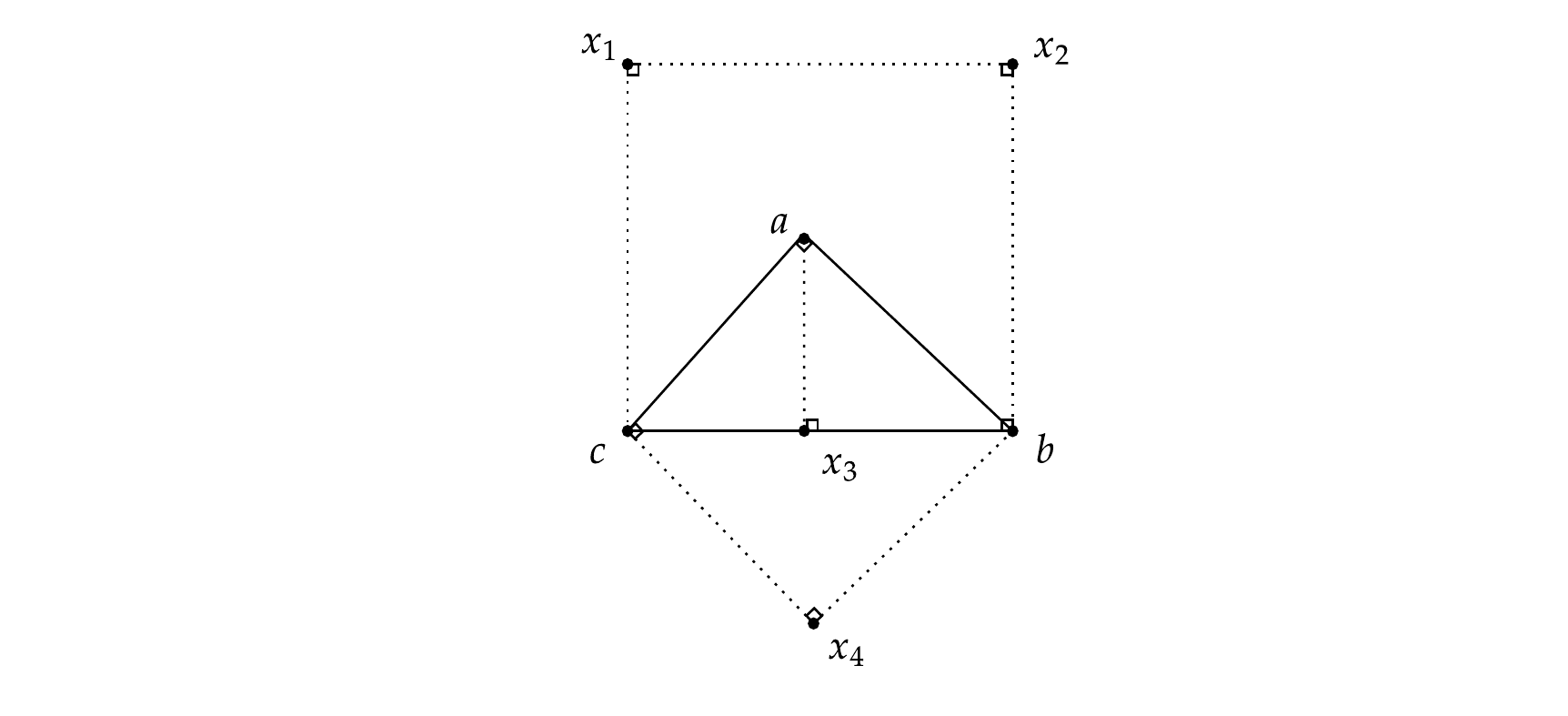}
\caption{Compatible Points with the Right Triangle.} \label{2.points} \end{figure}

Note that $\angle x_4ax_1, \angle x_4 a x_2 > \pi/2$. So, $x_4$ cannot be in the same point configuration as $x_1$ or $x_2$. By symmetry the same follows for $x_3$.  However, we may have both $x_1$ and $x_2$ or both $x_3$ and $x_4$, either of which give the unique optimal configuration $A$ in Figure \ref{fig: 3.thm}. 
\end{proof}

\begin{cor}\label{cor: p2-with-0}
    One might also wish to include the trivial 0-angle in our count.  In this case, $P(2)=4$, and the unique configuration is the square.
\end{cor}
\begin{proof}
    The only 5-point configuration no longer holds when we count the 0-angle.  Figure \ref{2.points} displays all valid four point configurations which define only 2 angles excluding 0, as detailed in the proof of $P(2)$.  All the shown points but $x_2$ define a 0-angle, so the only valid 4 point configuration is the square.
\end{proof}

\section{Computing $P(3) = 5$} \label{sec: P(3)}
\begin{lemma}\label{lem: convex quad 3 angles}
Let $ABCD$ be a convex quadrilateral defining three angles or fewer. Then, they form one of the following three configurations:
\begin{center}

\tikzset{every picture/.style={line width=0.75pt}} 

\begin{tikzpicture}[x=0.75pt,y=0.75pt,yscale=-1,xscale=1]

\draw   (443.8,160.21) -- (470.45,77.27) -- (557.97,77.04) -- (585.05,159.84) -- cycle ;
\draw    (513.6,83.44) -- (513.57,71.44) ;
\draw    (462.69,117.58) -- (452.68,113.6) ;
\draw    (576.68,115.28) -- (566.69,119.31) ;
\draw    (513.78,153.44) -- (513.82,166.44) ;
\draw  [fill={rgb, 255:red, 0; green, 0; blue, 0 }  ,fill opacity=1 ] (583.17,159.85) .. controls (583.17,160.78) and (584.02,161.52) .. (585.05,161.52) .. controls (586.09,161.52) and (586.93,160.77) .. (586.93,159.84) .. controls (586.93,158.91) and (586.08,158.16) .. (585.05,158.17) .. controls (584.01,158.17) and (583.17,158.92) .. (583.17,159.85) -- cycle ;
\draw  [fill={rgb, 255:red, 0; green, 0; blue, 0 }  ,fill opacity=1 ] (441.92,160.22) .. controls (441.93,161.14) and (442.77,161.89) .. (443.81,161.89) .. controls (444.84,161.89) and (445.68,161.14) .. (445.68,160.21) .. controls (445.68,159.28) and (444.83,158.53) .. (443.8,158.54) .. controls (442.76,158.54) and (441.92,159.29) .. (441.92,160.22) -- cycle ;
\draw  [fill={rgb, 255:red, 0; green, 0; blue, 0 }  ,fill opacity=1 ] (556.09,77.04) .. controls (556.09,77.97) and (556.94,78.72) .. (557.97,78.71) .. controls (559.01,78.71) and (559.85,77.96) .. (559.85,77.03) .. controls (559.85,76.11) and (559,75.36) .. (557.97,75.36) .. controls (556.93,75.36) and (556.09,76.11) .. (556.09,77.04) -- cycle ;
\draw  [fill={rgb, 255:red, 0; green, 0; blue, 0 }  ,fill opacity=1 ] (468.57,77.27) .. controls (468.57,78.2) and (469.42,78.95) .. (470.45,78.94) .. controls (471.49,78.94) and (472.33,78.19) .. (472.33,77.26) .. controls (472.32,76.33) and (471.48,75.59) .. (470.44,75.59) .. controls (469.41,75.59) and (468.57,76.34) .. (468.57,77.27) -- cycle ;
\draw   (299.36,119.96) -- (343.73,196.8) -- (255,196.8) -- cycle ;
\draw   (299.36,119.96) -- (388.09,119.96) -- (343.73,196.8) -- cycle ;
\draw    (299.36,190.49) -- (299.36,203) ;
\draw    (272.06,157.5) -- (281.16,164.33) ;
\draw    (317.56,164.33) -- (327.8,156.36) ;
\draw    (359.65,156.36) -- (373.3,163.19) ;
\draw    (344.86,126.79) -- (344.86,112) ;
\draw  [dash pattern={on 0.84pt off 2.51pt}]  (586.93,159.84) -- (515.03,210) ;
\draw  [dash pattern={on 0.84pt off 2.51pt}]  (515.03,210) -- (441.92,160.22) ;
\draw  [fill={rgb, 255:red, 0; green, 0; blue, 0 }  ,fill opacity=1 ] (256.88,196.8) .. controls (256.88,195.87) and (256.04,195.12) .. (255,195.12) .. controls (253.96,195.12) and (253.12,195.87) .. (253.12,196.8) .. controls (253.12,197.73) and (253.96,198.48) .. (255,198.48) .. controls (256.04,198.48) and (256.88,197.73) .. (256.88,196.8) -- cycle ;
\draw  [fill={rgb, 255:red, 0; green, 0; blue, 0 }  ,fill opacity=1 ] (301.24,119.96) .. controls (301.24,119.04) and (300.4,118.29) .. (299.36,118.29) .. controls (298.33,118.29) and (297.48,119.04) .. (297.48,119.96) .. controls (297.48,120.89) and (298.33,121.64) .. (299.36,121.64) .. controls (300.4,121.64) and (301.24,120.89) .. (301.24,119.96) -- cycle ;
\draw  [fill={rgb, 255:red, 0; green, 0; blue, 0 }  ,fill opacity=1 ] (345.6,196.8) .. controls (345.6,195.87) and (344.76,195.12) .. (343.72,195.12) .. controls (342.69,195.12) and (341.85,195.87) .. (341.85,196.8) .. controls (341.85,197.73) and (342.69,198.48) .. (343.72,198.48) .. controls (344.76,198.48) and (345.6,197.73) .. (345.6,196.8) -- cycle ;
\draw  [fill={rgb, 255:red, 0; green, 0; blue, 0 }  ,fill opacity=1 ] (389.97,119.96) .. controls (389.97,119.04) and (389.12,118.29) .. (388.09,118.29) .. controls (387.05,118.29) and (386.21,119.04) .. (386.21,119.96) .. controls (386.21,120.89) and (387.05,121.64) .. (388.09,121.64) .. controls (389.12,121.64) and (389.97,120.89) .. (389.97,119.96) -- cycle ;
\draw   (31,114) -- (192,114) -- (192,202) -- (31,202) -- cycle ;
\draw  [fill={rgb, 255:red, 0; green, 0; blue, 0 }  ,fill opacity=1 ] (32.88,114) .. controls (32.88,113.07) and (32.04,112.32) .. (31,112.32) .. controls (29.96,112.32) and (29.12,113.07) .. (29.12,114) .. controls (29.12,114.93) and (29.96,115.68) .. (31,115.68) .. controls (32.04,115.68) and (32.88,114.93) .. (32.88,114) -- cycle ;
\draw  [fill={rgb, 255:red, 0; green, 0; blue, 0 }  ,fill opacity=1 ] (193.88,114) .. controls (193.88,113.07) and (193.04,112.32) .. (192,112.32) .. controls (190.96,112.32) and (190.12,113.07) .. (190.12,114) .. controls (190.12,114.93) and (190.96,115.68) .. (192,115.68) .. controls (193.04,115.68) and (193.88,114.93) .. (193.88,114) -- cycle ;
\draw  [fill={rgb, 255:red, 0; green, 0; blue, 0 }  ,fill opacity=1 ] (32.88,202) .. controls (32.88,201.07) and (32.04,200.32) .. (31,200.32) .. controls (29.96,200.32) and (29.12,201.07) .. (29.12,202) .. controls (29.12,202.93) and (29.96,203.68) .. (31,203.68) .. controls (32.04,203.68) and (32.88,202.93) .. (32.88,202) -- cycle ;
\draw  [fill={rgb, 255:red, 0; green, 0; blue, 0 }  ,fill opacity=1 ] (193.88,202) .. controls (193.88,201.07) and (193.04,200.32) .. (192,200.32) .. controls (190.96,200.32) and (190.12,201.07) .. (190.12,202) .. controls (190.12,202.93) and (190.96,203.68) .. (192,203.68) .. controls (193.04,203.68) and (193.88,202.93) .. (193.88,202) -- cycle ;

\draw (504,227) node [anchor=north west][inner sep=0.75pt]   [align=left] {$\displaystyle 1.c$};
\draw (297,233) node [anchor=north west][inner sep=0.75pt]   [align=left] {$\displaystyle 1.b$};
\draw (98,233) node [anchor=north west][inner sep=0.75pt]   [align=left] {$\displaystyle 1.a$};

\end{tikzpicture}

\end{center}
where $1.a$ is a rectangle, $1.b$ is two attached equilateral triangles, and $1.c$ is four of the five vertices of a regular pentagon.
\end{lemma}
\begin{proof}
Assume we are not in case $1.a$, so that the angles of the quadrilateral are not all $\pi/2$. In particular, there is at least one obtuse angle, $\gamma$, and one acute angle, $\beta$. Any angle $\alpha$ formed by splitting $\beta$ is less than $\beta$ and thus must be exactly $\beta/2$ so as not to create two additional angles for a total of four. These three angles $\alpha=\beta/2,\beta,\gamma$ are then exactly the three angles in the configuration. Now we consider each of the four cases of placing $\beta$ and $\gamma$ about the quadrilateral, with the first listed angle corresponding to vertex $A$, second to $B$, and so on and with $A,B, C,$ and $D$ in clockwise cyclic order.
\begin{description}
\item[Case $\gamma\beta\gamma\beta$]
Equal opposite angles implies the quadrilateral is a parallelogram. The fact that $BD$ bisects the two $\beta$ angles implies that $ABCD$ is in fact a rhombus. Thus, $AC$ also bisects the $\gamma$ angles, implying that $\gamma/2 = \beta$. So, $6\beta = 2\pi$ and $\alpha = \pi/6, \beta = \pi/3,$ and $\gamma = 2\pi/3$ in this case. Given that is a rhombus, the configuration in this case is similar to $1.b$.
\item[Case $\gamma\gamma\beta\beta$] In this case, we have $\gamma + \beta = \pi$ from the angle sum of the quadrilateral. This implies that $AB$ and $CD$ are parallel. So, by analyzing the alternate interior angles given by the transversal $AC$, we have $\gamma = \alpha + (\gamma - \alpha)$, where $\alpha = \angle CAB$ and $\gamma - \alpha = \angle CAD$. Thus, $\gamma - \alpha = \beta$ and  $3\beta/2 = \gamma$, so $\alpha = \pi/5, \beta = 2\pi/5,$ and $\gamma = 3\pi/5$. Then by considering isosceles triangles $DAB$ and $ABC$, we see that segments $\overline{DA}, \overline{AB},$ and $\overline{BC}$ are all of equal length. Thus, the configuration is similar to $1.c$ in this case.

\item[Case $\gamma\gamma\gamma\beta$]
Diagonal $BD$ bisects the angle $\beta$. But then, from $\triangle BCD$ and from $ABCD$ we have 
\begin{align}
    \gamma + 2 \alpha = \gamma + \beta &= \pi \\
    3\gamma + \beta &= 2\pi.
\end{align}
This implies $\gamma = \pi/2$, then contradicting our assumption that we are not in case $1.a$.

\item[Case $\beta\beta\beta\gamma$]
Diagonal $BD$ bisects the angle $\beta$ at $B$ implying that it also bisects the angle $\gamma$ at $D$ by analyzing the two resultant triangles. Then, $\gamma/2 = \beta$ and thus, from $3\beta + \gamma = 2\pi$, we have $\alpha = \pi/5, \beta = 2\pi/5,$ and $\gamma = 4\pi/5$. But then, since $AC$ must bisect the angles of $\beta$ at $A$ and $C$, the angle sum of triangle $ABC$ is $2\alpha + \beta \neq \pi$, a contradiction.
\end{description}
\end{proof}
\begin{lemma} \label{lem: convex hull triangle}
Let $A, B, C, D,$ and $E$ be five points such that their convex hull is  $\triangle ABC$ and no four of them form a convex quadrilateral. Then, there are three possible classes of configurations:
\begin{enumerate}
    \item[2.a)] Points $D$ and $E$ are both on the same edge of $\triangle ABC$.
    \item[2.b)] One of $D$ and $E$ are on an edge of $\triangle ABC$ with the other on the segment joining the point on the edge to the opposite vertex.
    \item[2.c)] Both $D$ and $E$ are in the interior of $\triangle ABC$ and collinear with one of vertices $A$, $B$, or $C$.
\end{enumerate}

\end{lemma}
\begin{proof}
We proceed by casework on the number of $D$  and $E$ in the interior of $\triangle ABC$.
\begin{description}
\item[No points in the interior of $\triangle ABC$]
If neither $D$ nor $E$ are in the interior of $\triangle ABC$ then, since the convex hull of the five points is $\triangle ABC$, $D$ and $E$ must both be on the edges of $\triangle ABC$. If they are not on the same side of triangle, then the quadrilateral formed by $D$, $E$ and the ends of the edge which neither $D$ nor $E$ lie on is convex, yielding a contradiction. So, in this case the points are configuration $2.a)$.
\item[One point in the interior of $\triangle ABC$]
Suppose without loss of generality that $D$ is the point along an edge of $\triangle ABC$, say $\overline{AB}$. Then, $E$ is in the interior of $\triangle ABC$. Now $E$ must be on $\overline{CD}$ or else one of $ADEC$ or $BCED$ are a convex quadrilateral. Therefore, the points in this case are in configuration $2.b)$
\item[Both points in the interior of $\triangle ABC$]
Consider the three segments from $D$ to $A, B, $ and $C$. If $E$ is in the interior of $\triangle ADC$ then either $AEDB$ or $CEDB$ is a convex quadrilateral. This is similarly true if $E$ is in the interior of $\triangle ABD$ or $\triangle BCD$, so $E$ must be collinear with $D$ and one of $A, B,$ or $C$. Thus,  the points in this case are in configuration $2.c).$
\end{description}
\end{proof}

To deal with the configurations from Lemma \ref{lem: convex hull triangle}, we will need an additional lemma.

\begin{lemma} \label{lem: no pts in triangle}
Let $A,B,C, D$ be points such that $D$ is contained in the interior of $\triangle ABC$ and configuration induces at most three distinct angles. Then, $\triangle ABC$ must be equilateral and $D$ must be in the center of $\triangle ABC$. 
\end{lemma}
\begin{proof}
Note that $\angle ADB > \angle ACB, \angle ACD$. This is similarly true of $\angle BDC, \angle BAC, \angle BAD$ and of $\angle ADC, \angle ABC, \angle ABD$. Symmetry and the maximum of three distinct angles then allows the completion of all angles in the configuration:

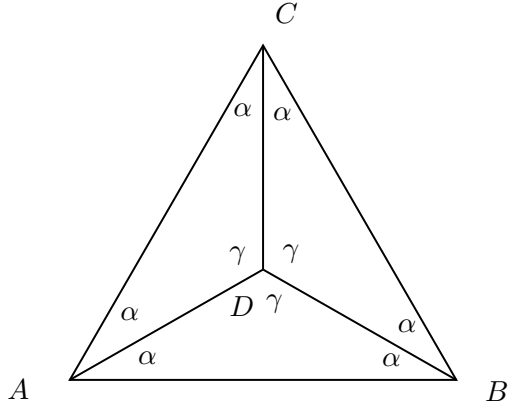
\begin{figure}[h!]
    \centering

\tikzset{every picture/.style={line width=0.75pt}} 

\begin{tikzpicture}[x=0.75pt,y=0.75pt,yscale=-1,xscale=1]

\draw   (323.46,69) -- (420.91,237.8) -- (226,237.8) -- cycle ;
\draw    (323.46,69) -- (323.46,182.1) ;
\draw    (323.46,182.1) -- (226,237.8) ;
\draw    (420.91,237.8) -- (323.46,182.1) ;

\draw (193,236) node [anchor=north west][inner sep=0.75pt]   [align=left] {$\displaystyle A$};
\draw (434,237) node [anchor=north west][inner sep=0.75pt]   [align=left] {$\displaystyle B$};
\draw (328,46) node [anchor=north west][inner sep=0.75pt]   [align=left] {$\displaystyle C$};
\draw (305,194) node [anchor=north west][inner sep=0.75pt]   [align=left] {$\displaystyle D$};
\draw (259,222) node [anchor=north west][inner sep=0.75pt]   [align=left] {$\displaystyle \alpha $};
\draw (250,200) node [anchor=north west][inner sep=0.75pt]   [align=left] {$\displaystyle \alpha $};
\draw (382,223) node [anchor=north west][inner sep=0.75pt]   [align=left] {$\displaystyle \alpha $};
\draw (390,206) node [anchor=north west][inner sep=0.75pt]   [align=left] {$\displaystyle \alpha $};
\draw (307,97) node [anchor=north west][inner sep=0.75pt]   [align=left] {$\displaystyle \alpha $};
\draw (327,99) node [anchor=north west][inner sep=0.75pt]   [align=left] {$\displaystyle \alpha $};
\draw (305,169) node [anchor=north west][inner sep=0.75pt]   [align=left] {$\displaystyle \gamma $};
\draw (332,168) node [anchor=north west][inner sep=0.75pt]   [align=left] {$\displaystyle \gamma $};
\draw (323.46,193.1) node [anchor=north west][inner sep=0.75pt]   [align=left] {$\displaystyle \gamma $};

\end{tikzpicture}

    \caption{Resultant Triangular Configuration}
    \label{fig:my_label}
\end{figure}

Thus, $\triangle ABC$ is equilateral and $D$ is in the center, as desired.

\end{proof}

Now we exhaustively check all possible configurations given by Lemmas \ref{lem: convex quad 3 angles} and \ref{lem: convex hull triangle}. 

\begin{description}
\item[1.a)] Consider adding a point to configuration $1.a$, with angles $\alpha \leq \beta < \gamma$ and $\alpha + \beta = \gamma = \pi/2$.  Then, if a point $E$ is added in the exterior of $ABCD$, it will form an obtuse angle with one edge of the angle being a side of the rectangle. For example, if $E$ is added below $\overline{CD}$, then $\angle BCE$ is obtuse. If $E$ is added to edge $AB$, then $\angle DEB$ is obtuse. It will similarly induce an obtuse angle if it is added to any other edge. Finally, if $E$ is added to the interior of $ABCD$, then the only way $E$ may be added without inducing an obtuse angle is if all the segments from $E$ to the vertices of the rectangle form angles of $\pi/2$ with each other at $E$. However, this would imply that the diagonals of $ABCD$ intersect at $E$ at a right angle, implying that $ABCD$ is a square. 

So, the only valid configurations require that $ABCD$ form a square. Moreover, if $ABCD$ form a square, we can still not induce any obtuse angles. This is because the other two angles in any triangle with an obtuse angle could not both be $\pi/4$ (and cannot be $\pi/2$), yielding more than three distinct angles. Thus, the maximal configuration in this case is adding a fifth point $E$ as the centerpoint of a square.
\item[1.b)]
In configuration $1.b)$, the angles are all determined: $\alpha = \pi/6, \beta = \pi/3$, and $\gamma = 2\pi/3$. Let the points $A$, $B$, $C$, and $D$ be in clockwise order around the configuration such that $\overline{AC}$ is the segment dividing the two equilateral triangles. 
In order to not contradict Lemma \ref{lem: convex quad 3 angles}, any added point must be in the interior of the rhombus (no point may be added to decrease the number of vertices in the convex hull since $ABCD$ is a parallelogram). In order for $E$ to not yield any angles smaller than $\alpha$, $E$ must be in the center of $\triangle ABC$ or $\triangle CDA$. However, in either case, this yields a new angle of $\pi/2$. So, no points may be added in this case.
\item[1.c)]
As in $1.b)$, the angles in $1.c)$ are all determined with $\alpha = \pi/5, \beta = 2\pi/5,$ and $\gamma = 3\pi/5$. Label the points $ABCD$ counterclockwise starting from the top left as in the diagram of $1.c$ in Lemma \ref{lem: convex quad 3 angles}. In order to not violate Lemma \ref{lem: convex quad 3 angles}, any added point must be in the interior of $ABCD$, must result in a triangular convex hull, or must be outside of $ABCD$ and have every convex quadrilateral in the configuration an instance of $1.c)$. In the former case, in order to not add an angle smaller than $\alpha$, $E$ must be added at the intersection of $\overline{AC}$ and $\overline{BE}$. In the second case, $E$ must be added at the intersection of $\dvec{AD}$ and $\dvec {BC}$. In the last case, the configuration with the added point cannot have a convex hull of a quadrilateral, as that quadrilateral could not be an instance of $1.c)$. Thus, it must be a pentagon. In order to guarantee that every convex quadrilateral in the configuration is a copy of $1.c)$, it must be regular. All three configurations are valid, but are not mutually compatible as adding multiple of these points would form an angle of magnitude less than $\alpha$.

\item[2.a)]
Three distinct angles are immediately induced in this case. Namely, $\angle ACD = \alpha < \angle ACE = \beta < \angle ACB = \gamma$. 
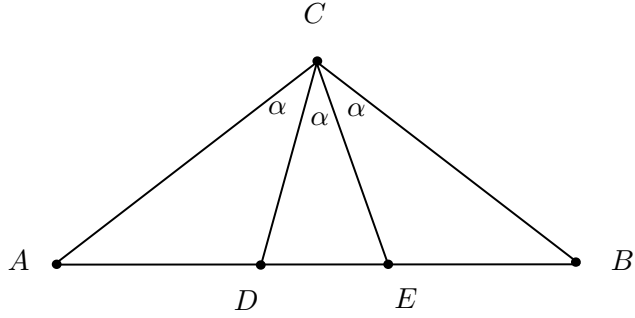
\begin{figure}[h!]
    \centering

\tikzset{every picture/.style={line width=0.75pt}} 

\begin{tikzpicture}[x=0.75pt,y=0.75pt,yscale=-1,xscale=1]

\draw    (183.88,183.61) -- (449,183.46) ;
\draw    (316.39,81.42) -- (449,183.46) ;
\draw    (183.88,183.61) -- (316.39,81.42) ;
\draw    (316.55,82.74) -- (352.47,184.02) ;
\draw    (316.39,81.42) -- (288.16,183.81) ;
\draw  [fill={rgb, 255:red, 0; green, 0; blue, 0 }  ,fill opacity=1 ] (448.93,182.34) .. controls (448.93,181.3) and (448.04,180.47) .. (446.95,180.47) .. controls (445.86,180.47) and (444.98,181.3) .. (444.98,182.34) .. controls (444.98,183.37) and (445.86,184.21) .. (446.95,184.21) .. controls (448.04,184.21) and (448.93,183.37) .. (448.93,182.34) -- cycle ;
\draw  [fill={rgb, 255:red, 0; green, 0; blue, 0 }  ,fill opacity=1 ] (187.07,183.45) .. controls (187.07,182.42) and (186.19,181.58) .. (185.1,181.58) .. controls (184.01,181.58) and (183.12,182.42) .. (183.12,183.45) .. controls (183.12,184.49) and (184.01,185.32) .. (185.1,185.32) .. controls (186.19,185.32) and (187.07,184.49) .. (187.07,183.45) -- cycle ;
\draw  [fill={rgb, 255:red, 0; green, 0; blue, 0 }  ,fill opacity=1 ] (290,183.81) .. controls (290,182.78) and (289.12,181.94) .. (288.03,181.94) .. controls (286.94,181.94) and (286.05,182.78) .. (286.05,183.81) .. controls (286.05,184.85) and (286.94,185.68) .. (288.03,185.68) .. controls (289.12,185.68) and (290,184.85) .. (290,183.81) -- cycle ;
\draw  [fill={rgb, 255:red, 0; green, 0; blue, 0 }  ,fill opacity=1 ] (354.28,183.45) .. controls (354.28,182.42) and (353.4,181.58) .. (352.31,181.58) .. controls (351.22,181.58) and (350.33,182.42) .. (350.33,183.45) .. controls (350.33,184.49) and (351.22,185.32) .. (352.31,185.32) .. controls (353.4,185.32) and (354.28,184.49) .. (354.28,183.45) -- cycle ;
\draw  [fill={rgb, 255:red, 0; green, 0; blue, 0 }  ,fill opacity=1 ] (318.53,80.87) .. controls (318.53,79.84) and (317.64,79) .. (316.55,79) .. controls (315.46,79) and (314.58,79.84) .. (314.58,80.87) .. controls (314.58,81.9) and (315.46,82.74) .. (316.55,82.74) .. controls (317.64,82.74) and (318.53,81.9) .. (318.53,80.87) -- cycle ;

\draw (290.29,99.8) node [anchor=north west][inner sep=0.75pt]   [align=left] {$\displaystyle \alpha $};
\draw (311.56,105.15) node [anchor=north west][inner sep=0.75pt]   [align=left] {$\displaystyle \alpha $};
\draw (330.17,100.92) node [anchor=north west][inner sep=0.75pt]   [align=left] {$\displaystyle \alpha $};
\draw (159,175) node [anchor=north west][inner sep=0.75pt]   [align=left] {$\displaystyle A$};
\draw (463,175) node [anchor=north west][inner sep=0.75pt]   [align=left] {$\displaystyle B$};
\draw (308,50) node [anchor=north west][inner sep=0.75pt]   [align=left] {$\displaystyle C$};
\draw (273,195) node [anchor=north west][inner sep=0.75pt]   [align=left] {$\displaystyle D$};
\draw (354,194) node [anchor=north west][inner sep=0.75pt]   [align=left] {$\displaystyle E$};

\end{tikzpicture}

    \caption{Case 2.a)}
    \label{fig:my_label}
\end{figure}

Since the difference between each pair of angles is also induced by this configuration, we have that $\beta = 2\alpha$ and $\gamma = 3\alpha$. Since $\angle ADC > \angle AEC > \angle ABC$, we have $\angle ADC = \gamma, \angle AEC = \beta, $ and $\angle ABC = \alpha$. This is similarly true of $\angle CEB, \angle CDB,$ and $\angle CAB$ by symmetry. Thus, the angle sum of $\triangle ACB$ implies $5 \alpha = \pi$ and thus $\alpha \pi/5, \beta = 2\pi/5, $ and $\gamma = 3\pi/5$. Thus this configuration is completely determined. 

Now, if another point were added, either the convex hull would remain a triangle or there would be four points who form a convex quadrilateral. In the former case, no point could be in the interior of a triangle, as that would force the angles to be as in Lemma \ref{lem: no pts in triangle}, which they are not. Thus, an additional added point would have to be placed on an existing edge. It could not be placed on $\overline{AB}$, as it would split an angle of $\alpha$. If it were placed on $\overline{AC}$ or $\overline{BC}$ it would form a convex quadrilateral. Given the induced values of the angles in this case, that quadrilateral would have to be similar to configuration $1.c)$. However, from the prior casework, no configuration containing a similar copy of $1.c)$ may have more than five points.
\item[2.b)] 
In this case, there is a vertex contained in a triangle. From Lemma \ref{lem: no pts in triangle}, this forces the triangle to be equilateral and the point to be in the center of the triangle. However, this induces angles of $\pi/6, \pi/3, 2\pi/3$. Moreover, $D$ and $E$ form a right angle, yielding more than three angles.
\item[2.c)]
In this case,  Lemma \ref{lem: no pts in triangle} requires that both interior points be in the center of $\triangle ABC$ simultaneously, a contradiction.
\end{description}

Therefore, $P(3) = P(2) = 5$, with five optimal configurations as in Figure \ref{fig: 3.thm}.

\begin{cor}\label{cor: p3-with-0}
    One might also wish to include the trivial 0-angle in our count.  In this case, $P(3)=5$, but the square with the center-point and the pentagon are now the only valid configurations.
\end{cor}
\begin{proof}
    The set of valid five-point configurations when we count the 0-angle must be a subset of the valid five-point configurations we identified above.  By direct inspection, the square with the center-point and the pentagon are the only of the five in Figure \ref{fig: 3.thm} which define only three angles. All the others define three angles greater than zero and also the 0-angle by collinearity.
\end{proof}

\section{Future Work}
While it seems possible to compute $P(k)$ by exhaustive casework for higher values of $k$, the casework quickly becomes overwhelming. Additionally, while it is potentially possible to repeat such methods in higher dimensions, the visualization of the proofs played a crucial role in this analysis. In combination with the added degrees of freedom from adding dimensions, this would make this method of computation quickly intolerable.

Future work may tighten our upper bound on $P(k)$. However, we make the following conjecture.
\begin{conj}
    The lower bound on $P(k)$ in Theorem \ref{thm: linear bounds on P(k)} is tight.  Namely, $P(2k)=2k+3$ and $P(2k+1)=2k+3$ for all $k\geq1$.
\end{conj}
Therefore, we believe that future work should improve the upper bound of $P(n)\leq 6n$, either via progress towards the Weak Dirac Conjecture (which would still fall short of our conjecture) or by some other means.  Alternatively, future research may find a more efficient method of constructing viable point sets without the need for the exhaustive search we perform.

It is also an open problem to investigate $P(k)$ with point sets in more than two dimensions. Low angle configurations using variations of Lenz's construction, as in \cite{GenPaper}, may yield insight into optimal structures in higher dimensions.





 \appendix

\end{document}